\documentclass[10pt]{elsarticle}

\usepackage{verbatim,amsopn,amsthm,amssymb,amsmath,graphicx,subfigure}
\usepackage{hyperref}

\theoremstyle{plain}
\newtheorem{thm}{Theorem}

\newtheorem{prop}[thm]{Proposition}
\newtheorem{cor}[thm]{Corollary}

\theoremstyle{definition}

\theoremstyle{remark}

\newcommand{\Cay}{\Gamma}

\begin{document}

\title{Primitivity of group rings of non-elementary torsion-free hyperbolic groups}
\author{Brent B. Solie}
\address{Department of Mathematics, Embry-Riddle Aeronautical University, 3700 Willow Creed Road, Prescott, AZ 86301}
\ead{solieb@erau.edu}
\date{\today}
\begin{keyword}primitive group rings \sep hyperbolic groups

\MSC[2010] 16S34 \sep 20C07 \sep 20F67 \end{keyword}

\begin{abstract}
  We use a recent result of Alexander and Nishinaka to show that if $G$ is a non-elementary torsion-free hyperbolic group and $R$ is a countable domain, then the group ring $RG$ is primitive. This implies that the group ring $KG$ of any non-elementary torsion-free hyperbolic group $G$ over a field $K$ is primitive.
\end{abstract}

\maketitle

\section{Introduction}

In \cite{Alexander2017}, Alexander and Nishinaka use the following Property (*) of a group $G$ to establish the primitivity of a broad class of group rings.
(Recall that a ring $R$ is \emph{(right) primitive} if it contains a faithful irreducible (right) $R$-module.)\\

\begin{minipage}{0.9\linewidth}
  (*) For each subset $M$ of $G$ consisting of a finite number of elements not equal to 1, and for any positive integer $m \geq 2$, there exist distinct $a,b,c \in G$ so that if $(x_1^{-1}g_1 x_1) (x_2^{-1}g_2 x_2) \cdots (x_m^{-1}g_m x_m) = 1$, where $g_i \in M$ and $x_i \in \{a, b, c\}$ for all $i=1, \dots, m$, then $x_i = x_{i+1}$ for some $i$.
\end{minipage}

In particular, Alexander and Nishinaka give the following broad criterion for primitivity:

\begin{thm}[{\cite[Theorem 1.1]{Alexander2017}}]
  \label{Alexander-Nishinaka}
  Let $G$ be a group which has a non-Abelian free subgroup whose cardinality is the same as that of $G$, and suppose that $G$ satisfies Property (*). Then, if $R$ is a domain with $|R| \leq |G|$, the group ring $RG$ of $G$ over $R$ is primitive. In particular, the group algebra $KG$ is primitive for any field $K$.
\end{thm}

Our goal at present is to show that the hypotheses of Theorem 1 hold if $G$ is a non-elementary torsion-free hyperbolic group, thus implying the primitivity of a fairly broad class of group rings.

\section{Hyperbolic groups}

Let $G$ be a group with finite generating set $X$.
Recall that the \emph{Cayley graph} $\Cay_X(G)$ of $G$ with respect to $X$ is an $X$-digraph with vertex set $G$ and an $x$-labelled edge directed from $g$ to $gx$ for all $g \in G$ and $x \in X$.
We may equip $\Cay_X(G)$ with the word metric by assigning each edge a length of one; as a result, $\Cay_X(G)$ becomes a geodesic metric space.

We say that $\Cay_X(G)$ satisfies the \emph{$\delta$-thin triangles condition} for a real constant $\delta >0$ if every side of a geodesic triangle is contained in the $\delta$-neighborhood of the union of the two remaining sides.
If a group $G$ admits a finite generating set $X$ and real constant $\delta > 0$ for which $\Cay_X(G)$ satisfies the $\delta$-thin triangles condition, we say that $G$ is \emph{$\delta$-hyperbolic}, or simply \emph{hyperbolic}.
Hyperbolicity does not depend on choice of finite generating set, although the constant of hyperbolicity may vary.

It is easy to see that finite groups and free groups are hyperbolic.
Some more interesting examples of hyperbolic groups include one-relator groups with torsion, certain small cancellation groups, and fundamental groups of closed surfaces with negative Euler characteristic.
On the other hand, free Abelian groups of rank at least two and Baumslag-Solitar groups are standard examples of non-hyperbolic groups.

The intrinsic geometry of hyperbolic groups induces a wide range of algebraic, geometric, and algorithmic properties.
Hyperbolic groups are necessarily finitely presented, have a solvable word problem, have linear isoperimetric inequality, and satisfy the Tits alternative.
Furthermore, hyperbolicity is a common phenomenon among groups: standard statistical models in group theory show that a randomly chosen finitely presented group is overwhelmingly likely to be infinite, torsion-free, and hyperbolic \cite{Olshanskii1992}.
Hyperbolic groups and their generalizations are the subject of a tremendous amount of modern literature in group theory, and so we refer the interested reader to \cite{Alonso1991, Ghys1990, Gromov1987} for a more thorough introduction.

We say that a hyperbolic group is \emph{elementary} if it is finite or virtually infinite cyclic; in either case, primitivity of a group ring over such a group has already been characterized.
Fundamental to these characterizations is the \emph{FC center} of a group $G$, the set of elements $\Delta(G) = \{g \in G \mid [G:C_G(g)] < \infty\}$.
Similarly, the \emph{FC+ center} of $G$ is the subset $\Delta^+(G)$ of torsion elements of $\Delta(G)$.

It is well-known that if $KG$ is primitive, then $\Delta^+(G) = 1$ \cite[Lemma 9.1.1]{Passman1977}.
Consequently, if $G$ is finite, then $KG$ is not primitive for any field $K$ unless $G=1$.
Alternately, if $G$ is virtually infinite cyclic, it is necessarily polycyclic-by-finite and therefore covered by the classification collectively due to Domanov \cite{Domanov1978}, Farkas-Passman \cite{Farkas1978}, and Roseblade \cite{Roseblade1978}: for any field $K$ and polycyclic-by-finite $G$, the group ring $KG$ is primitive if and only if $\Delta(G)=1$ and $K$ is not absolute.

We now restrict our attention to non-elementary torsion-free hyperbolic groups.
We begin with a theorem of Gromov.

\begin{thm}[{\cite[Theorem 5.3.E]{Gromov1987}}]
  \label{Gromov's Theorem}
  There exists a constant $E=E(k,\delta)>0$ such that for every $k$ elements $g_1, \dots, g_k$ of infinite order in a $\delta$-hyperbolic group $G$, the subgroup $\langle g_1^{e_1}, \dots, g_k^{e_k} \rangle$ is free whenever $e_i \geq E$ for $i=1, \dots, k$.
\end{thm}

We obtain the following immediate corollary.

\begin{cor}
  \label{Exists a noncommuting element}
  Let $G$ be a non-elementary torsion-free hyperbolic group.
  Let $M = \{g_1, \dots, g_k\}$ be a finite subset of nonidentity elements of $G$.
  Then there exists an element $u \in G$ which is not a proper power and which commutes with no element of $M$.
\end{cor}

\begin{proof}
  Let $C(g) = C_G(g)$ denote the centralizer of $g$ in $G$.
  It is well-known that centralizers of nontrivial elements of a torsion-free hyperbolic group are necessarily infinite cyclic \cite{Ghys1990}.
  Consequently, if $g \in G$ is nontrivial and not a proper power, then $C(g) = C(g^n) = \langle g \rangle$ for all $n \neq 0$ and $C(g)$ is maximal among the cyclic subgroups of $G$.
  Thus we may replace $g_i^n \in M$ with $g_i$ without changing the set of elements which commute with no element of $M$, and we may assume that no element of $M$ is a proper power.
  The same argument shows that we may always produce $u$ which is not a proper power by passing to a root if necessary.

  Suppose $k=1$.
  As $G$ is non-elementary, it must properly contain the maximal cyclic subgroup $C(g_1)$, and so there must exist an element $u$ of $G$ failing to commute with $g_1$.

  Now suppose $k \geq 2$.
  Let $E=E(k,\delta)$ be the constant from Theorem \ref{Gromov's Theorem}; we may assume $E$ to be a positive integer.
  Since $G$ is torsion-free, $g_i^E$ is nontrivial and of infinite order for all $g_i \in M$.
  Therefore $H=\langle g_1^E, \dots, g_k^E \rangle$ is a non-Abelian free subgroup of $G$ and so contains an element $u \in H \subseteq G$ which fails to commute with any $g_i^E$.
  Since $C(g_i^E)=C(g_i)$ in $G$, we have that $u$ cannot commute with any $g_i$.
\end{proof}

In order to show that a non-elementary torsion-free hyperbolic group satisfies Property (*), we resort to a version of the big powers property of torsion-free hyperbolic groups.
The form we use here follows immediately from the stronger version for a class of relatively hyperbolic groups proved in \cite{Kharlampovich2009}.

\begin{thm}[The big powers property \cite{Kharlampovich2009}]
  Let $G$ be a torsion-free hyperbolic group.
  Let $u \in G$ be nontrivial and not a proper power.
  Let $g_1, \dots, g_k$ be elements of $G$ which do not commute with $u$.
  Then there exists $N > 0$ such that if $|n_i| \geq N$ for $i=0, \dots, k$ then
  \begin{align*}
    u^{n_0} g_1 u^{n_1} g_2 \cdots u^{n_{k-1}} g_k u^{n_k} \neq 1.
  \end{align*}
\end{thm}

The big powers property provides a way of generating a large set of nontrivial elements of a group and so is of use in studying residual properties of groups, universal equivalence, and algebraic geometry over groups.
The property appears first due to B. Baumslag in his study of fully residually free groups \cite{Baumslag1967}.
Ol'shanski\u\i\ later generalizes the property to torsion-free hyperbolic groups \cite{Olshanskii1993}, and Kharlampovich and Myasnikov further generalize it to non-Abelian torsion-free relatively hyperbolic groups with free Abelian parabolic subgroups \cite{Kharlampovich2009}.

\section{Main Result}

\begin{prop}
  \label{Property (*) for NETFHG}
  If $G$ is a non-elementary torsion-free hyperbolic group, then $G$ satisfies Property (*).
\end{prop}

\begin{proof}
  Let $M$ be a finite subset of $G$ not containing the identity.
  By Corollary \ref{Exists a noncommuting element}, there exists an element $u \in G$ which generates its own centralizer and commutes with no $g \in M$.

  Fix a positive integer $m \geq 2$ and consider a finite sequence $g_1, \dots, g_m$ of elements from $M$.
  Since $u$ commutes with none of the $g_i$ and $u$ generates its own centralizer, by the big powers property, there exists $N(g_1, \dots, g_m) > 0$ such that
  \begin{align*}
    u^{n_0} g_1 u^{n_1} g_2 \cdots g_{m-1} u^{n_{m-1}} g_m u^{n_m} \neq 1
  \end{align*}
  whenever $|n_i| \geq N$ for all $i=0, \dots, m$.

  Since $M$ is a finite set, there are finitely many $m$-tuples $(g_1, \dots, g_m)$ drawn from $M$.
  Therefore, let $N > \max \left\{N(g_1, \dots, g_m) \mid g_1, \dots, g_m \in M \right\}$.

  We now define $a=u^N, b=u^{2N}$, and $c=u^{3N}$.
  Since $G$ is torsion-free, these elements are necessarily distinct.
  Consider a product
  \begin{align*}
    w = (x_1^{-1}g_1 x_1) (x_2^{-1}g_2 x_2) \cdots (x_m^{-1}g_m x_m)
  \end{align*}
  where $x_1, \dots, x_m \in \{ u^N, u^{2N}, u^{3N}\}$.
  We then have
  \begin{align*}
    w = u^{n_0} g_1 u^{n_1} g_2 \cdots g_{m-1} u^{n_{m-1}} g_m u^{n_m},
  \end{align*}
  where $u^{n_0} = x_1^{-1}, u^{n_m} = x_m, u^{n_i} =x_i x_{i+1}^{-1}$ and $n_i \in \{0, \pm N, \pm 2N\}$ for $i = 1, \dots, m-1$.
  Note that by choice of $x_1$ and $x_m$, we have $n_0 \neq 0$ and $n_m \neq 0$.

  By the big powers property and choice of $N$, if $n_i \neq 0$ for all $i=0, \dots, m$, then $w \neq 1$.
  Therefore, if $w=1$, then some $n_i = 0$.
  Since we cannot have $n_0=0$ or $n_{m+1}=0$, we have $n_i=0$ for some $i \in \{1, \dots, m-1\}$, in which case we must have $1 = u^{n_i} = x_i x_{i+1}^{-1} $, and so $x_i = x_{i+1}$.
\end{proof}

Since a non-elementary hyperbolic group is finitely generated by definition, it is necessarily countably infinite.
However, by \cite[Remark 3.6]{Alexander2017}, a countably infinite group satisfying Property (*) also necessarily contains a free subgroup of rank two, which is also countably infinite.
(We also note that a group $G$ satisfying Property (*) also automatically satisfies $\Delta(G)=1$.)
A non-elementary torsion-free hyperbolic group thus satisfies Property (*) and has a free subgroup of the same cardinality, and so we obtain the following main result as a corollary to Theorem \ref{Alexander-Nishinaka}.

\begin{thm}
  \label{the result}
  If $G$ is a non-elementary torsion-free hyperbolic group, then for any countable domain $R$, the group ring $RG$ of $G$ over $R$ is primitive. In particular, the group ring $KG$ is primitive for any field $K$.
\end{thm}

We may slightly relax the torsion-free condition on $G$ if $\Delta(G)=1$, as demonstrated by the following corollary.

\begin{cor}
  If $G$ is a non-elementary virtually torsion-free hyperbolic group with $\Delta(G)=1$, then for any countable domain $R$, the group ring $RG$ of $G$ over $R$ is primitive. In particular, the group ring $KG$ is primitive for any field $K$.
\end{cor}

\begin{proof}
  Let $K$ be any field and let $H$ be a torsion-free subgroup of finite index in $G$.
  As a finite index subgroup, $H$ is quasi-isometric to $G$, and thus $H$ is necessarily also non-elementary and hyperbolic \cite{Gromov1987}.
  By Theorem \ref{the result}, we have that $KH$ is primitive.

  From \cite[Theorem 4.2.10]{Passman1977}, KG is prime if and only if $\Delta(G)$ is torsion-free Abelian.
  Thus if $\Delta(G)=1$, we have that $KG$ is necessarily prime.

  Finally, if $KG$ is prime, then $KG$ is primitive if and only if $KH$ is primitive by a result of Rosenberg \cite[Theorem 3]{Rosenberg1971} (cf. \cite[Theorem 9.1.11]{Passman1977}.)
\end{proof}

It is worth noting the long-standing open conjecture that every hyperbolic group is virtually torsion-free.
An affirmation of this conjecture, together with the above corollary, would imply that $KG$ is primitive for any field $K$ and any non-elementary hyperbolic group $G$.

\section*{Acknowlegements}

The author would like to thank Prof. Tsunekazu Nishinaka for introducing the author to the subject of primitive group rings and for many illuminating conversations, Prof. Hisaya Tsutsui for his support and guidance, and the referee for his or her very helpful feedback.

\section*{References}

\bibliographystyle{plain}
\bibliography{prim}

\begin{thebibliography}{10}

\bibitem{Alexander2017}
James Alexander and Tsunekazu Nishinaka.
\newblock Non-noetherian groups and primitivity of their group algebras.
\newblock {\em J. Algebra}, 473:221--246, 2017.

\bibitem{Alonso1991}
J.~M. Alonso, T.~Brady, D.~Cooper, V.~Ferlini, M.~Lustig, M.~Mihalik,
  M.~Shapiro, and H.~Short.
\newblock Notes on word hyperbolic groups.
\newblock In {\em Group theory from a geometrical viewpoint ({T}rieste, 1990)},
  pages 3--63. World Sci. Publ., River Edge, NJ, 1991.
\newblock Edited by Short.

\bibitem{Baumslag1967}
Benjamin Baumslag.
\newblock Residually free groups.
\newblock {\em Proc. London Math. Soc. (3)}, 17:402--418, 1967.

\bibitem{Domanov1978}
O.~I. Domanov.
\newblock Primitive group algebras of polycyclic groups.
\newblock {\em Sibirsk. Mat. \v Z.}, 19(1):37--43, 236, 1978.

\bibitem{Farkas1978}
Daniel~R. Farkas and D.~S. Passman.
\newblock Primitive {N}oetherian group rings.
\newblock {\em Comm. Algebra}, 6(3):301--315, 1978.

\bibitem{Ghys1990}
\'Etienne Ghys and Pierre de~la Harpe.
\newblock La propri\'et\'e de {M}arkov pour les groupes hyperboliques.
\newblock In {\em Sur les groupes hyperboliques d'apr\`es {M}ikhael {G}romov
  ({B}ern, 1988)}, volume~83 of {\em Progr. Math.}, pages 165--187.
  Birkh\"auser Boston, Boston, MA, 1990.

\bibitem{Gromov1987}
M.~Gromov.
\newblock Hyperbolic groups.
\newblock In {\em Essays in group theory}, volume~8 of {\em Math. Sci. Res.
  Inst. Publ.}, pages 75--263. Springer, New York, 1987.

\bibitem{Kharlampovich2009}
O.~Kharlampovich and A.~Myasnikov.
\newblock Limits of relatively hyperbolic groups and lyndon's completions.
\newblock arXiv:0904.2423v3 [math.GR], 2009.

\bibitem{Olshanskii1992}
A.~Yu. Ol'shanski\u\i.
\newblock Almost every group is hyperbolic.
\newblock {\em Internat. J. Algebra Comput.}, 2(1):1--17, 1992.

\bibitem{Olshanskii1993}
A.~Yu. Ol'shanski\u\i.
\newblock On residualing homomorphisms and {$G$}-subgroups of hyperbolic
  groups.
\newblock {\em Internat. J. Algebra Comput.}, 3(4):365--409, 1993.

\bibitem{Passman1977}
Donald~S. Passman.
\newblock {\em The algebraic structure of group rings}.
\newblock Pure and Applied Mathematics. Wiley-Interscience [John Wiley \&\
  Sons], New York-London-Sydney, 1977.

\bibitem{Roseblade1978}
J.~E. Roseblade.
\newblock Prime ideals in group rings of polycyclic groups.
\newblock {\em Proc. London Math. Soc. (3)}, 36(3):385--447, 1978.

\bibitem{Rosenberg1971}
Alan Rosenberg.
\newblock On the primitivity of the group algebra.
\newblock {\em Canad. J. Math.}, 23:536--540, 1971.

\end{thebibliography}

\end{document}